\documentclass[12pt,a4paper,psamsfonts]{amsart}
\usepackage{amssymb,amscd,amsxtra,calc}
\usepackage{cmmib57}
\usepackage[matrix,arrow,curve]{xy}

\newtheorem{thm}{Theorem}[section]
\newtheorem{lemma}[thm]{Lemma}
\newtheorem{proposition}[thm]{Proposition}
\newtheorem{corollary}[thm]{Corollary}
\newtheorem{claim}[thm]{Claim}

\newtheorem{conjecture}[thm]{Conjecture}

\theoremstyle{definition}
\newtheorem{definition}[thm]{Definition}
\newtheorem{remark}[thm]{Remark}
\newtheorem{notation}[thm]{Notation}
\newtheorem{example}[thm]{Example}

\newcommand{\pr}{\mathbb{P}}

\newcommand{\Z}{\mathbb{Z}}
\newcommand{\Q}{\mathbb{Q}}

\newcommand{\C}{\mathbb{C}}

\newcommand{\NC}{\operatorname{N}_1}
\newcommand{\ND}{\operatorname{N}^1}
\newcommand{\NE}{\operatorname{NE}}

\newcommand{\Exc}{\operatorname{Exc}}
\newcommand{\Sing}{\operatorname{Sing}}

\newcommand{\Pic}{\operatorname{Pic}}

\newcommand{\rank}{\operatorname{rank}}

\newcommand{\sO}{\mathcal{O}}
\newcommand{\sN}{\mathcal{N}}
\newcommand{\sE}{\mathcal{E}}

\newcommand{\sX}{\mathcal{X}}
\newcommand{\sL}{\mathcal{L}}

\newcommand{\sH}{\mathcal{H}}

\newcommand{\M}{\operatorname{M}}
\newcommand{\GM}{\operatorname{GM}}
\newcommand{\LM}{\operatorname{LM}}
\newcommand{\LGM}{\operatorname{LGM}}

\title{The Mukai conjecture for log Fano manifolds}
\author{Kento Fujita}

\begin{document}
\maketitle
\begin{abstract}{\noindent For a log Fano manifold $(X, D)$ with $D\neq 0$ and with 
the log Fano pseudoindex $\geq 2$, we prove that the restriction homomorphism 
$\Pic(X)\rightarrow\Pic(D_1)$ of Picard groups is injective for any irreducible 
component $D_1\subset D$.
The strategy of our proof is to run a certain minimal model program 
and is similar to the argument of Casagrande's one. 
As a corollary, we prove that the Mukai conjecture 
(resp.\ the generalized Mukai conjecture) 
implies the log Mukai conjecture (resp.\ the log generalized Mukai conjecture). }
\end{abstract}

\section{Introduction}\label{introsection}

In the previous paper \cite{logfano}, we considered a \emph{log Fano manifold} $(X, D)$, 
that is, $X$ is a smooth projective variety and $D$ is a reduced 
simple normal crossing divisor on $X$ with $-(K_X+D)$ ample. 
We got several classification results in \cite{logfano}, especially the result 
related to the \emph{Mukai conjecture} (see \cite{mukaiconj}) 
and the \emph{generalized Mukai conjecture} (see \cite{BCDD}). 

\begin{conjecture}[Mukai conjecture ($\M^n_\rho$)]\label{Mconj}
Fix $n$, $\rho\in\Z_{>0}$. 
Let $X$ be an $n$-dimensional Fano manifold with the Fano index $r$ which satisfies 
that $\rho(X)\geq\rho$ and $r\geq (n+\rho)/\rho$. Then $\rho(X)=\rho$, 
$r=(n+\rho)/\rho$ and 
$X\simeq(\pr^{r-1})^\rho$ holds. 
\end{conjecture}

\begin{conjecture}[generalized Mukai conjecture ($\GM^n_\rho$)]\label{GMconj}
Fix $n$, $\rho\in\Z_{>0}$. 
Let $X$ be an $n$-dimensional Fano manifold with the Fano pseudoindex $\iota$ 
which satisfies 
that $\rho(X)\geq\rho$ and $\iota\geq (n+\rho)/\rho$. Then $\rho(X)=\rho$, $\iota=(n+\rho)/\rho$ and 
$X\simeq(\pr^{\iota-1})^\rho$ holds. 
\end{conjecture}

We proved a special version of the log versions of the Mukai conjecture and 
the generalized Mukai conjecture in \cite[Theorem 4.3]{logfano}; we call them the 
\emph{log Mukai conjecture} and the 
\emph{log generalized Mukai conjecture} respectively. 
(In \cite[Theorem 4.3]{logfano}, we proved Conjecture $\LGM_2^n$.)

\begin{conjecture}[log Mukai conjecture ($\LM^n_\rho$)]\label{LMconj}
Fix $n$, $\rho\geq 2$. 
Let $(X, D)$ be an $n$-dimensional log Fano manifold with the log Fano index $r$ 
and $D\neq 0$ which satisfies that 
$\rho(X)\geq\rho$ and $r\geq (n+\rho-1)/\rho$. Then $\rho(X)=\rho$, 
$r=(n-1+\rho)/\rho$ and 
$(X, D)$ is isomorphic to the case of Type $(\rho, r; m_1,\dots,m_{\rho-1})$ 
with $m_1\dots,m_{\rho-1}\geq 0$ in Example \ref{lm}. 
\end{conjecture}

\begin{conjecture}[log generalized Mukai conjecture ($\LGM^n_\rho$)]\label{LGMconj}
Fix $n$, $\rho\geq 2$. 
Let $(X, D)$ be an $n$-dimensional log Fano manifold with the log Fano 
pseusoindex $\iota$ and $D\neq 0$ which satisfies that 
$\rho(X)\geq\rho$ and $\iota\geq (n+\rho-1)/\rho$. 
Then $\rho(X)=\rho$, $\iota=(n-1+\rho)/\rho$ and 
$(X, D)$ is isomorphic to the case of Type $(\rho, \iota; m_1,\dots,m_{\rho-1})$ 
with $m_1\dots,m_{\rho-1}\geq 0$ in Example \ref{lm}. 
\end{conjecture}

\begin{remark}\label{intrormk}
Clearly, Conjecture $\GM^n_\rho$ (resp.\ Conjecture $\LGM^n_\rho$) implies 
Conjecture $\M^n_\rho$ 
(resp.\ Conjecture $\LM^n_\rho$) (see Remark \ref{pseudo_rmk}).
We also note that Conjecture $\GM^n_\rho$ is true if $n\leq 5$ (\cite{ACO}) or 
$\rho\leq 3$ (\cite{NO}). 
In \cite[Proposition 4.1]{logfano}, we also showed that Conjecture $\LGM^2_\rho$ 
is true (see also \cite[\S 3]{Maeda}). 
\end{remark}

In this article, we obtain a fundamental property to compare the Picard number 
of $X$ and $D$ for a log Fano manifold $(X, D)$. 

\begin{thm}[= Theorem \ref{mainthm}]\label{mainthm_intro}
Let $(X, D)$ be an $n$-dimensional log Fano manifold with $D\neq 0$. 
Then one of the following holds: 
\begin{enumerate}
\renewcommand{\theenumi}{\arabic{enumi}}
\renewcommand{\labelenumi}{(\theenumi)}
\item
The restriction homomorphism $\Pic(X)\rightarrow\Pic(D)$ is injective. 
\item
$X$ admits a $\pr^1$-bundle structure $\pi\colon X\rightarrow Y$ and $D$ is a section 
of $\pi$. In particular, $D$ is irreducible and isomorphic to $Y$ $($hence $Y$ is an 
$(n-1)$-dimensional Fano manifold$)$. 
\end{enumerate}
\end{thm}

Especially, for a log Fano manifold $(X, D)$ with $D\neq 0$ and the log Fano 
pseudoindex $\geq 2$, 
we get a comparison theorem of  the Picard number of $X$ and $D_1\subset D$.

\begin{corollary}[= Corollary \ref{piccor} \eqref{piccor1}]\label{piccor_intro}
Let $(X, D)$ be a log Fano manifold with $D\neq 0$ and 
the log Fano pseudoindex $\geq 2$. Then the restriction homomorphism 
$\Pic(X)\rightarrow\Pic(D_1)$ is injective 
for any irreducible component $D_1\subset D$.
\end{corollary}

To prove Theorem \ref{mainthm_intro}, we use the result of \cite{BCHM} that 
$X$ is a \emph{Mori dream space} (see \cite{HK} for the definition) 
for a log Fano manifold $(X, D)$. 
We run a special $(-D)$-minimal model program (\emph{MMP}, for short) and compare 
the cokernel of the homomorphism 
$\NC(D)\rightarrow\NC(X)$ in each step of the MMP. We can show that 
the dimension of the cokernel is constant by using the same way of Casagrande's one 
\cite{cas1,cas2}. 

As a corollary, we can show that the Mukai Conjecture 
(resp.\ the generalized Mukai Conjecture) 
implies the log Mukai Conjecture (resp.\ the generalized log Mukai Conjecture). 

\begin{thm}[= Theorem \ref{mukai_log}]\label{mukai_log_intro}
Conjectures $\M^n_\rho$ and $\LM^n_\rho$ 
$($resp.\ Conjectures $\GM^n_\rho$ and $\LGM^n_\rho$$)$ imply 
Conjecture $\LM^{n+1}_\rho$ $($resp.\ Conjecture $\LGM^{n+1}_\rho$$)$ 
for any $n$, $\rho\geq 2$. 
\end{thm}

Using this theorem, we obtain the following corollary immediately. 

\begin{corollary}[= Corollary \ref{lgm_cor} \eqref{lgm_cor1}]\label{rhothree_intro}
Let $(X, D)$ be an $n$-dimensional log Fano manifold with the log Fano 
pseusoindex $\iota$ and $D\neq 0$ which satisfies that 
$\rho(X)\geq 3$ and $\iota\geq (n+2)/3>1$. Then $\iota=(n+2)/3$ and 
$(X, D)$ is isomorphic to the case of Type $(3, r; m_1,m_2)$ 
with $m_1, m_2\geq 0$ in Example \ref{lm}. That is, 
\begin{eqnarray*}
X & \simeq & \pr_{\pr^{\iota-1}\times\pr^{\iota-1}}
(\sO^{\oplus\iota}\oplus\sO(m_1,m_2))\\
D & \simeq & \pr_{\pr^{\iota-1}\times\pr^{\iota-1}}(\sO^{\oplus\iota})
\end{eqnarray*}
with $m_1, m_2\geq 0$, where the embedding is obtained by the canonical projection under these isomorphisms. 
\end{corollary}

\smallskip

\noindent\textbf{Acknowledgements.}
The author got the main idea (Section \ref{run}) of this article 
during the participation of Singularity Seminar in Nihon University. 
He thanks Professor Kei-ichi Watanabe, the organizer of the seminar. 
The author is partially supported by JSPS Fellowships for Young Scientists. 

\bigskip

\noindent\textbf{Notation and terminology.}
There is no difference between the notation of this and the previous 
paper \cite{logfano}. 

We always work in the category of algebraic 
(separated and finite type) schemes over a fixed algebraically closed field $\Bbbk$ 
of characteristic zero. 
A \emph{variety} means a connected and reduced algebraic scheme. 
For a variety $X$, the set of singular points on $X$ is denoted by $\Sing(X)$. 

For the theory of extremal contraction, we refer the readers to \cite{KoMo}. 
For a complete variety $X$, the Picard number of $X$ is denoted by  $\rho(X)$. 
For a complete variety $X$ and a closed subscheme $D$ on $X$, 
the image of the homomorphism $\NC(D)\rightarrow\NC(X)$ is denoted by 
$\NC(D, X)$. 
For a smooth projective variety $X$ and a $K_X$-negative extremal ray $R\subset\overline{\NE}(X)$,
\[
l(R):=\min\{(-K_X\cdot C)\mid C\text{ is a rational curve with } [C]\in R\}
\]
is called the \emph{length} $l(R)$ of $R$. 
A rational curve $C\subset X$ with 
$[C]\in R$ and $(-K_X\cdot C)=l(R)$ is called \emph{a minimal rational curve of $R$}.

For a morphism of algebraic schemes $f\colon X\rightarrow Y$, we define the 
\emph{exceptional locus} $\Exc(f)$ \emph{of $f$} by 
\[
\Exc(f):=\{x\in X\mid f \text{ is not isomorphism around }x\}.
\]

For algebraic schemes $X_1,\dots,X_m$, the projection is denoted by 
$p_{i_1,\dots,i_k}\colon\prod_{i=1}^mX_i\rightarrow\prod_{j=1}^kX_{i_j}$ for 
any $1\leq i_1<\cdots<i_k\leq m$. 

For an algebraic scheme $X$ and a locally free sheaf of finite rank $\sE$ on $X$, 
let $\pr_X(\sE)$ be the projectivization of $\sE$ in the sense of Grothendieck 
and $\sO_\pr(1)$ be the tautological invertible sheaf. We usually denote 
the projection by $p\colon\pr_X(\sE)\rightarrow X$. 

We write $\sO(m_1,\dots,m_s)$ on $\pr^{n_1}\times\dots\times\pr^{n_s}$ instead of 
$p_1^*\sO_{\pr^{n_1}}(m_1)\otimes\cdots\otimes p_s^*\sO_{\pr^{n_s}}(m_s)$ 
for simplicity.

\section{Log Fano manifolds}

We recall the definitions and some properties of log Fano manifolds and 
snc Fano varieties quickly. For more informations, see \cite[Section 2]{logfano}. 

\begin{definition}[snc Fano variety, log Fano manifold {\cite[Definition 2.9]{logfano}}]
\label{logFano_dfn}
\begin{enumerate}
\renewcommand{\theenumi}{\roman{enumi}}
\renewcommand{\labelenumi}{(\theenumi)}
\item\label{logFano_dfn1}
A \,\, variety $\sX$ is called an $n$-dimensional 
\emph{simple normal crossing} (\emph{snc,} for short) 
\emph{Fano variety} if 
$\sX$ is an equi-$n$-dimensional projective variety having 
normal crossing singularities (that is, \[
\sO_{\sX,x}\simeq\Bbbk[[x_1,\dots,x_{n+1}]]/(x_1\cdots x_k)
\]
holds 
for some $1\leq k\leq n+1$, for any closed point $x\in\sX$), 
each irreducible component $X$ of $\sX$ is smooth 
and $\omega_{\sX}^\vee$ (the dual of the dualizing sheaf) is ample. 
\item\label{logFano_dfn2}
An $n$-dimensional \emph{log Fano manifold} is a pair $(X, D)$ such that 
$X$ is an $n$-dimensional smooth projective 
variety and $D$ is a reduced and simple normal crossing divisor on $X$ 
(that is, $D$ has normal crossing singularities 
and each irreducible component of $D$ is smooth) 
with $-(K_X+D)$ ample. 
\end{enumerate}
\end{definition}

\begin{definition}[index {\cite[Definition 2.11]{logfano}}]\label{index_dfn}
\begin{enumerate}
\renewcommand{\theenumi}{\roman{enumi}}
\renewcommand{\labelenumi}{(\theenumi)}
\item\label{index_dfn1}
Let $\sX$ be an snc Fano variety. We define the \emph{snc Fano index} of $\sX$ as 
\[
\max\{r\in\Z_{>0}\mid\omega_{\sX}^{\vee}\simeq\sL^{\otimes r}
\text{ for some }\sL\in\Pic(\sX)\}.
\]
\item\label{index_dfn2}
Let $(X, D)$ be a log Fano manifold. We define the \emph{log Fano index} 
of $(X, D)$ as 
\[
\max\{r\in\Z_{>0}\mid -(K_X+D)\sim rL
\text{ for some Cartier divisor }L\text{ on }X\}.
\]
\end{enumerate}
\end{definition}

\begin{definition}[pseudoindex {\cite[Definition 2.12]{logfano}}]\label{ps_dfn}
\begin{enumerate}
\renewcommand{\theenumi}{\roman{enumi}}
\renewcommand{\labelenumi}{(\theenumi)}
\item\label{ps_dfn1}
Let $\sX$ be an snc Fano variety. We define the 
\emph{snc Fano pseudoindex} of $\sX$ as 
\[
\min\{\deg_C(\omega_{\sX}^{\vee}|_C)\mid C\subset\sX
\text{ rational curve}\}.
\]
\item\label{ps_dfn2}
Let $(X, D)$ be a log Fano manifold. We define the \emph{log Fano pseudoindex} 
of $(X, D)$ as 
\[
\min\{(-(K_X+D)\cdot C)\mid C\subset X
\text{ rational curve}\}.
\]
\end{enumerate}
\end{definition}

\begin{remark}[{\cite[Remark 2.13]{logfano}}]\label{pseudo_rmk}
For an snc Fano variety $\sX$ (resp.\ a log Fano manifold $(X, D)$), 
the snc Fano pseudoindex (resp.\ the log Fano pseudoindex) $\iota$ is divisible by the 
snc Fano index (resp.\ the log Fano index) $r$ by definition. 
In particular, $\iota\geq r$ holds. 
\end{remark}

\begin{definition}[{conductor divisor \cite[Definition 2.4]{logfano}}]\label{cond_dfn}
Let $\sX$ be an snc Fano variety with the irreducible decomposition 
$\sX=\bigcup_{1\leq i\leq m}X_i$. For any distinct $1\leq i$, $j\leq m$, 
the intersection $X_i\cap X_j$ can be seen as a smooth divisor $D_{ij}$ in $X_i$. 
We define 
\[
D_i:=\sum_{j\neq i}D_{ij}
\]
and call it the \emph{conductor divisor} in $X_i$ (with respect to $\sX$). 
We often write that 
$(X_i, D_i)\subset\sX$ is an irreducible component 
for the sake of simplicity. 
We also write $\sX=\bigcup_{1\leq i\leq m}(X_i, D_i)$ 
for emphasizing the conductor divisors. 
\end{definition}

\begin{remark}[{\cite[Remark 2.14]{logfano}}]\label{irr_rmk}
Let $\sX$ be an $n$-dimensional snc Fano variety with the snc Fano index $r$, 
the snc Fano pseudoindex $\iota$. 
Then the log Fano index of $(X, D)$ is divisible by $r$ and the log Fano pseudoindex 
of $(X, D)$ is at least $\iota$, where $(X, D)\subset\sX$ is an irreducible component 
with the conductor divisor.
\end{remark}

Now, we show several important properties for log Fano manifolds and snc Fano varieties 
without proofs. See \cite{logfano} for proofs. 

\begin{thm}[{\cite[Theorem 2.20 (1)]{logfano}}]\label{fujino_thm}
Let $(X,D)$ be a log Fano manifold such that 
the log Fano index is divisible by $r$ $($resp.\ the log Fano pseudo index $\geq\iota$$)$. 
Then $D$ is a $($connected$)$ snc Fano variety and the snc Fano index 
is also divisible by $r$ $($resp.\ the snc Fano pseudoindex $\geq\iota$$)$. 
\end{thm}

\begin{proposition}[{\cite[Proposition 2.8, Theorem 2.20 (2)]{logfano}}]\label{fanopicglue}
Let $\sX$ be an $n$-dimensional snc Fano variety 
with the irreducible decomposition $\sX=\bigcup_{i=1}^mX_i$. 
We also let $X_{ij}:=X_i\cap X_j$ $($scheme theoretic intersection$)$ for any 
$1\leq i<j\leq m$. 
Then we have an exact sequence 
\[
0\rightarrow\Pic(\sX)\xrightarrow{\eta}\bigoplus_{i=1}^m\Pic(X_i)
\xrightarrow{\mu}\bigoplus_{1\leq i<j\leq m}\Pic(X_{ij}),
\]
where $\eta$ is the restriction homomorphism and 
\[
\mu\Bigl((\sH_i)_i\Bigr):=(\sH_i|_{X_{ij}}\otimes\sH_j^\vee|_{X_{ij}})_{i<j}.
\]
\end{proposition}

\begin{lemma}[{\cite[Lemma 2.16]{logfano}, \cite[Corollary 2.2, Lemma 2.3]{Maeda}}]\label{tor_free}
Let $(X,D)$ be a log Fano manifold. 
Then $\Pic(X)$ is torsion free. Furthermore if $\Bbbk=\C$, 
the homomorphism 
\[
\Pic(X)\rightarrow H^2(X^\text{\rm{an}}; \Z)
\]
is isomorphism. 
\end{lemma}

The following result is most essential in this article. 

\begin{thm}[{\cite[Corollary 1.3.2]{BCHM}, \cite[Theorem 2.24]{logfano}}]\label{dream}
Let $(X, D)$ be a log Fano manifold. Then $X$ is a Mori dream space. 
\end{thm}

\section{Running a minimal model program}\label{run}

In this section, we consider a special minimal model program for a log Fano manifold, 
whose argument is similar to Casagrande's one \cite{cas1,cas2}. 

First, we see Ishii's result. 

\begin{lemma}[{\cite[Lemma 1.1]{ishii}, (cf. \cite[Theorem 2.2]{cas1}})]\label{ishii_lem}
Let $Y$ be a projective variety with canonical singularities. 
Let $R\subset\overline{\NE}(Y)$ be an extremal ray such that 
the contraction morphism $\pi\colon Y\rightarrow Z$ associated to $R$ is 
of birational type, and let $E:=\Exc(\pi)$. Assume that each fiber of 
the restriction morphism $\pi|_E\colon E\rightarrow\pi(E)$ to its image is of dimension 
one. Then each fiber of $\pi|_E$ is a union of smooth rational curves and 
$0<(-K_Y\cdot l)\leq 1$ for a component $l$ of a fiber of $\pi|_E$ which contains 
a Gorenstein point of $Y$. 
\end{lemma}

We recall that we can run a $B$-MMP for any $\Q$-divisor $B$ 
for a Mori dream space. 

\begin{proposition}[{\cite[Proposition 1.11 (1)]{HK}, \cite[Proposition 2.2]{cas2}}]\label{mmp_prop}
Let $X$ be a Mori dream space and $B$ be a $\Q$-divisor on $X$. 
Then there exists a sequence of birational maps among normal, $\Q$-factorial and 
projective varieties 
\begin{equation*}\label{mmp_seq}
X=X^0\stackrel{\sigma_0}{\dashrightarrow}X^1\stackrel{\sigma_1}{\dashrightarrow}
\dots\stackrel{\sigma_{k-1}}{\dashrightarrow}X^k
\end{equation*}
and a $\Q$-divisor $B^i$ on $X^i$ for any $0\leq i\leq k$ such that 
\begin{enumerate}
\renewcommand{\theenumi}{\roman{enumi}}
\renewcommand{\labelenumi}{(\theenumi)}
\item
The birational map $\sigma_i$ is decomposed into the following diagram 
\[\xymatrix{
X^i \ar@{-->}[rr]^{\sigma_i} \ar[dr]_{\pi_i} &  & X^{i+1} \ar[dl]^{\pi_i^+}\\
& Y^i & \\
}\]
and $B^i$ is the strict transform of $B$ on $X^i$ for any $0\leq i\leq k-1$. 
\item
The morphism $\pi_i$ is the birational contraction morphism associated to an 
extremal ray $R^i\subset\NE(X^i)$ such that $(B^i\cdot R^i)<0$ and 
$\pi_i^+$ is the flip of $\pi_i$ $($if $\pi_i$ is small$)$ or the identity morphism 
$($if $\pi_i$ is divisorial$)$ for any $0\leq i\leq k-1$.
\item
Either $B^k$ is nef on $X^k$ or there exists a fiber type extremal contraction 
$X^k\xrightarrow{\pi_k}Y^k$ associated to the extremal ray $R^k\subset\NE(X^k)$ 
such that $(B^k\cdot R^k)<0$ holds. 
\end{enumerate}
We call this step by a \emph{$B$-minimal model program} $($a 
\emph{$B$-MMP}, for short$)$. 
\end{proposition}

For a log Fano manifold $(X, D)$, the smooth 
projective variety $X$ is a Mori dream space 
by Theorem \ref{dream}. Hence we can apply Proposition \ref{mmp_prop}. 
Moreover, we can choose a $B$-MMP which is also a $(K_X+D)$-MMP. 
The proof is completely same as that of \cite[Proposition 2.4]{cas2} 
(replacing $-K_X$ with $-(K_X+D)$). We omit a proof. 

\begin{proposition}[cf. {\cite[Proposition 2.4]{cas2}, \cite[Remark 3.10.10]{BCHM}}]\label{both_mmp}
Let $(X, D)$ be a log Fano manifold and $B$ be a $\Q$-divisor on $X$. 
Then we can choose a $B$-MMP which is also a $(K_X+D)$-MMP. 
\end{proposition}

We are in particular interested in the case where $B$ is equal to $-D$. 

\begin{notation}\label{logfano_mmp}
Let $(X, D)$ be an $n$-dimensional log Fano manifold with $D\neq 0$. 
We set the irreducible decomposition $D=\sum_{i=1}^mD_i$. 
We consider a $(-D)$-MMP (as in 
Proposition \ref{mmp_prop}) which is also a $(K_X+D)$-MMP as in 
Proposition \ref{both_mmp} 
(we note that this is also a $K_X$-MMP). 
We set $D_i^j$ such as the strict transform of $D_i$ in $X^j$ for any 
$1\leq i\leq m$ and $0\leq j\leq k$. 
Let $A^1\subset X^1$ be the indeterminancy locus of $\sigma_0^{-1}$, 
and for $2\leq j\leq k$, 
let $A^j\subset X^j$ be the union of the strict transform of $A^{j-1}\subset X^{j-1}$, 
with the indeterminancy locus of $\sigma_{j-1}^{-1}$. 
\end{notation}

The next lemma is essentially established by Casagrande \cite{cas1}. 
For a proof, see \cite[Lemma 3.8]{cas1}. 

\begin{lemma}[cf. {\cite[Lemma 3.8]{cas1}}]\label{negativity}
Under Notation \ref{logfano_mmp}, we have the following properties:
\begin{enumerate}
\renewcommand{\theenumi}{\arabic{enumi}}
\renewcommand{\labelenumi}{(\theenumi)}
\item\label{negativity1}
For any $1\leq j\leq k$, the dimension of $A^j$ is at most $n-2$,  
$X^j\setminus A^j$ is isomorphic to an open subscheme of $X$ and 
\[
\Sing(X^j)\subset A^j\subset D^j
\]
holds. Moreover, $\dim A^j>0$ whenever $\pi_{j-1}$ is small. 
\item\label{negativity2}
For any $1\leq j\leq k$, $X^j$ has terminal singularities and 
the pair $(X^j, D^j)$ is a dlt pair. 
Moreover, if $C\subset X^j$ is an irreducible curve not contained in $A^j$ and 
$C^0\subset X$ its strict transform, we have 
\[
(-(K_{X^j}+D^j)\cdot C)\geq(-(K_X+D)\cdot C^0),
\]
with the strict inequality whenever $C\cap A^j\neq\emptyset$.
\end{enumerate}
\end{lemma}

The next proposition is the key of this article.

\begin{proposition}[cf. {\cite[Lemma 2.6]{cas2}}]\label{fund_prop}
Under Notation \ref{logfano_mmp}, we have the following properties:
\begin{enumerate}
\renewcommand{\theenumi}{\arabic{enumi}}
\renewcommand{\labelenumi}{(\theenumi)}
\item\label{fund_prop1}
For any $0\leq j\leq k$, the divisor $D^j$ is nonzero effective. 
In particular, this MMP ends with a fiber type contraction. That is, 
there exists a fiber type extremal contraction 
$X^k\xrightarrow{\pi_k}Y^k$ associated to the extremal ray $R^k\subset\NE(X^k)$ 
such that $(D^k\cdot R^k)>0$ and $((K_{X^k}+D^k)\cdot R^k)<0$ holds. 
The restriction morphism $\pi_k|_{D^k}\colon D^k\rightarrow Y^k$ is surjective. 
\item\label{fund_prop2}
The restriction morphism $\pi_j|_{D_i^j}\colon D_i^j\rightarrow\pi_j(D_i^j)$ to its image 
is an algebraic fiber space, that is, $(\pi_j|_{D_i^j})_*\sO_{D_i^j}=\sO_{\pi_j(D_i^j)}$, 
for any $1\leq i\leq m$ and $0\leq j\leq k$. 
\item\label{fund_prop3}
There exists an irreducible curve $C^j\subset D^j$ such that $\pi_j(C^j)$ is a point 
for any $0\leq j\leq k-1$. 
\item\label{fund_prop4}
If the restriction morphism $\pi_k|_{D^k}\colon D^k\rightarrow Y^k$ is a finite morphism, 
then $k=0$ and the morphism $(\pi_0=)\pi_k\colon X^k\rightarrow Y^k$ 
is a $\pr^1$-bundle and $(D=)D^k$ is a section of $\pi_k$. 
\item\label{fund_prop5}
If the log Fano pseudoindex $\iota$ of $(X, D)$ satisfies $\iota\geq 2$, then 
$\dim Y^k\leq n-2$ holds. 
\end{enumerate}
\end{proposition}

\begin{proof}
\eqref{fund_prop1}
We prove by induction on $j$ to prove that the divisor $D^j$ is nonzero effective. 
The case $j=0$ is trivial. Assume that $j\geq 1$ and $D^{j-1}$ is nonzero effective. 
We assume that $D^j$ is not nonzero effective. 
Then $D^{j-1}$ is a prime divisor and $\pi_{j-1}$ is a divisorial contraction 
which contracts $D^{j-1}$, but this leads to a contradiction 
since $(D^{j-1}\cdot R^{j-1})>0$. Thus $D^j$ is nonzero effective for any 
$0\leq j\leq k$. Since $D^k$ is nonzero effective, $-D^k$ cannot be nef. 
Therefore this MMP ends with a fiber type contraction. 
We also know that the restriction morphism $\pi_k|_{D^k}\colon D^k\rightarrow Y^k$ 
is surjective since any fiber and $D^k$ intersect with each other. 

\eqref{fund_prop2}
It is enough to show that the homomorphism 
$(\pi_j)_*\sO_{X^j}\rightarrow(\pi_j|_{D_i^j})_*\sO_{D_i^j}$ is surjective. 
We know that the sequence 
\[
(\pi_j)_*\sO_{X^j}\rightarrow(\pi_j|_{D_i^j})_*\sO_{D_i^j}
\rightarrow R^1(\pi_j)_*\sO_{X^j}(-D_i^j)
\]
is exact. Since the pair $(X^j, D^j)$ is a $\Q$-factorial dlt pair 
by Lemma \ref{negativity} \eqref{negativity2}, 
we know that the pair $(X^j, \sum_{i'\neq i}D_{i'}^j)$ is also a dlt pair by 
\cite[Corollary 2.39]{KoMo}. 
Since $-D_i^j-(K_{X^j}+\sum_{i'\neq i}D_{i'}^j)=-(K_{X^j}+D^j)$ is $\pi_j$-ample, 
we have $R^1(\pi_j)_*\sO_{X^j}(-D_i^j)=0$ by a vanishing theorem 
(see for example \cite[Theorem 2.42]{fujino}). Therefore 
the restriction morphism $\pi_j|_{D_i^j}\colon D_i^j\rightarrow\pi_j(D_i^j)$ to its image 
is an algebraic fiber space. 

\eqref{fund_prop3}
Assume that the restriction morphism $\pi_j|_{D^j}\colon D^j\rightarrow Y^j$ is 
a finite morphism for some $0\leq j\leq k-1$. 
Let $F^j$ ba an arbitrary nontrivial fiber of $\pi_j$. 
Then $F^j$ and $D^j$ intersect with each other since $(D^j\cdot R^j)>0$. 
If $\dim F^j\geq 2$, then $\dim(F^j\cap D^j)\geq 1$
since $D^j$ is a $\Q$-Cartier divisor. This is a contradiction to the assumption 
that $\pi_j|_{D^j}$ is a finite morphism. 
Therefore $\dim F^j=1$ for any nontrivial fiber of $\pi_j$. 
Let $l^j\subset F^j$ be an arbitrary irreducible component. Then $l^j\not\subset A^j$ 
since $A^j\subset D^j$ by Lemma \ref{negativity} \eqref{negativity1}, 
and $(D^j\cdot l^j)>0$ by the property $(D^j\cdot R^j)>0$. 
Hence we can apply Lemma \ref{ishii_lem}; we have $(-K_{X^j}\cdot l^j)\leq 1$. 
Let $l\subset X$ be the strict transform of $l^j\subset X^j$. Then 
\[
(-(K_X+D)\cdot l)\leq(-(K_{X^j}+D^j)\cdot l^j)=(-K_{X^j}\cdot l^j)-(D^j\cdot l^j)<1
\]
holds by Lemma \ref{negativity} \eqref{negativity2}. This leads to a contradiction 
since $-(K_X+D)$ is an ample Cartier divisor. Therefore 
the restriction morphism $\pi_j|_{D^j}\colon D^j\rightarrow Y^j$ is not 
a finite morphism for any $0\leq j\leq k-1$. 

\eqref{fund_prop4}
We have $\dim Y^k=n-1$ by \eqref{fund_prop1}. 
If there exists a fiber $F^k\subset X^k$ of $\pi_k$ such that $\dim F^k\geq 2$, 
then $\dim(D^k\cap F^k)\geq 1$ holds. This leads to a contradiction 
since $\pi_k|_{D^k}$ is a finite morphism. Thus any fiber of $\pi_k$ is of dimension one. 
We can take a general smooth fiber $l^k\subset X^k$ of $\pi_k$ such that 
$l^k\cap A^k=\emptyset$. Since $(-(K_{X^k}+D^k)\cdot R^k)>0$, $(D^k\cdot R^k)>0$ 
and $l^k\cap\Sing(X^k)=\emptyset$ 
(hence $D^k$ and $K_{X^k}$ is Cartier around $l^k$), 
we have $l^k\simeq\pr^1$, $(-K_{X^k}\cdot l^k)=2$ and $(D^k\cdot l^k)=1$. 
We assume that $k\geq 1$. Then $A^k\neq\emptyset$ holds. 
Let $l_0^k\subset X^k$ be a fiber of $\pi_k$ such that $l_0^k\cap A^k\neq\emptyset$ 
holds. We know that $(-(K_{X^k}+D^k)\cdot l_0^k)=1$ by \cite[Theorem 1.3.17]{kollar}. 
We note that any arbitrary irreducible component $l_1^k$ of $l_0^k$ satisfies 
$l_1^k\not\subset A^k$ since $l^k\not\subset D^k$and 
$A^k\subset D^k$ holds by Lemma \ref{negativity} \eqref{negativity1}. 
Let $l_1^k\subset l_0^k$ be an irreducible component such that 
$l_1^k\cap A^k\neq\emptyset$ holds, and let $l_1\subset X$ be the strict transform 
of $l_1^k\subset X^k$. Then we have 
\[
(-(K_X+D)\cdot l_0)<(-(K_{X^k}+D^k)\cdot l_0^k)\leq 1
\]
by Lemma \ref{ishii_lem}. However, this leads to a contradiction 
since $-(K_X+D)$ is an ample Cartier divisor. 
Hence $k=0$ holds. 
Thus the morphism $\pi_0=\pi_k\colon X\rightarrow Y^0$ satisfies that 
$\dim F^0=1$ for any fiber of $\pi_0$, and for general fiber $l^0\subset X$, 
we have $(-K_X\cdot l^0)=2$ and $(D\cdot l^0)=1$. 
Therefore $\pi_0$ is a $\pr^1$-bundle and $D$ is a section of $\pi_0$ by 
\cite[Lemma 2.12]{fujita}. 

\eqref{fund_prop5}
Assume that $\dim Y^k=n-1$. Then a general fiber $l^k\subset X^k$ of $\pi^k$ 
satisfies that $l^k\cap A^k=\emptyset$, $l^k\simeq\pr^1$ and $(-K_{X^k}\cdot l^k)=2$ 
by the same argument of the proof of \eqref{fund_prop4}. Thus we have 
\[
(-(K_X+D)\cdot l)\leq(-(K_{X^k}+D^k)\cdot l^k)<2,
\]
where $l\subset X$ is the strict transform of $l^k\subset X^k$, 
by Lemma \ref{ishii_lem} and the property $(D^k\cdot l^k)>0$. 
This contradict to the property $\iota\geq 2$. Therefore $\dim Y^k\leq n-2$ holds. 
\end{proof}

\begin{corollary}[cf. {\cite[Lemma 3.6]{cas1}}]\label{codim_const}
Under Notation \ref{logfano_mmp}, we have the following results: 
\begin{enumerate}
\renewcommand{\theenumi}{\arabic{enumi}}
\renewcommand{\labelenumi}{(\theenumi)}
\item\label{codim_const1}
The equality $\rho(X)-\dim\NC(D, X)=\rho(X^j)-\dim\NC(D^j, X^j)$ 
holds for any $0\leq j\leq k$. 
\item\label{codim_const2}
We have $\rho(X)-\dim\NC(D, X)=0$ or $1$. If $\rho(X)-\dim\NC(D, X)=1$, 
then $k=0$, the morphism $\pi_0\colon X\rightarrow Y^0$ 
is a $\pr^1$-bundle and $D$ is a section of $\pi_0$. 
\end{enumerate}
\end{corollary}

\begin{proof}
\eqref{codim_const1}
We prove by induction on $j$. The case $j=0$ is obvious. 
We consider the case $1\leq j\leq k$. It is enough to show the equality 
$\rho(X^{j-1})-\dim\NC(D^{j-1}, X^{j-1})=\rho(X^j)-\dim\NC(D^j, X^j)$. 
We know that $\dim\NC(\pi_{j-1}(D^{j-1}), Y^{j-1})=\dim\NC(D^{j-1}, X^{j-1})-1$ by 
Proposition \ref{fund_prop} \eqref{fund_prop3}. 

If $\pi_{j-1}$ is small, then any curve in $X^j$ that is contracted by $\pi_{j-1}^+$
is in $D^j$ since $-D^j$ is $(\pi_{j-1}^+)$-ample. 
Hence $\dim\NC(\pi_{j-1}(D^{j-1}), Y^{j-1})=\dim\NC(D^j, X^j)-1$. 
Therefore  $\rho(X^{j-1})-\dim\NC(D^{j-1}, X^{j-1})=\rho(X^j)-\dim\NC(D^j, X^j)$ 
holds since $\rho(X^{j-1})=\rho(X^j)$. 

If $\pi_{j-1}$ is divisorial, then $\sigma_{j-1}=\pi_{j-1}$ and 
$\rho(X^j)=\rho(X^{j-1})-1$ holds. Therefore 
$\rho(X^{j-1})-\dim\NC(D^{j-1}, X^{j-1})=\rho(X^j)-\dim\NC(D^j, X^j)$ holds. 

\eqref{codim_const2}
The value $\rho(X^k)-\dim\NC(D^k, X^k)$ is equal to $0$ or $1$ 
since the restriction morphism $\pi_k|_{D^k}\colon D^k\rightarrow Y^k$ is surjective 
and the dimension of the kernel of the surjection 
$(\pi_k)_*\colon\NC(X^k)\rightarrow\NC(Y^k)$ is one. 
If $\rho(X^k)-\dim\NC(D^k, X^k)=1$, then the restriction homomorphism 
$\NC(D^k, X^k)\rightarrow\NC(Y^k)$ is isomorphism. Thus any curve in $D^k$ 
cannot be contracted. Hence the assertion holds by Proposition \ref{fund_prop} 
\eqref{fund_prop4}. 
\end{proof}

As an immediate corollary, we get the following theorem. 

\begin{thm}[Main Theorem]\label{mainthm}
Let $(X, D)$ be an $n$-dimensional log Fano manifold with $D\neq 0$. 
Then one of the following holds: 
\begin{enumerate}
\renewcommand{\theenumi}{\arabic{enumi}}
\renewcommand{\labelenumi}{(\theenumi)}
\item\label{mainthm1}
The restriction homomorphism $\Pic(X)\rightarrow\Pic(D)$ is injective. 
\item\label{mainthm2}
$X$ admits a $\pr^1$-bundle structure $\pi\colon X\rightarrow Y$ and $D$ is a section 
of $\pi$. In particular, $D$ is irreducible and isomorphic to $Y$ $($hence $Y$ is an 
$(n-1)$-dimensional Fano manifold and the log Fano pseudoindex of $(X, D)$ is equal to one$)$. 
\end{enumerate}
\end{thm}

\begin{proof}
If $\rho(X)-\dim\NC(D, X)=1$, then \eqref{mainthm2} holds 
by Corollary \ref{codim_const} \eqref{codim_const2}. 
If $\rho(X)-\dim\NC(D, X)=0$, then the homomorphism 
$\NC(D)\rightarrow\NC(X)$ is surjective. Hence the dual homomorphism 
$\ND(X)\rightarrow\ND(D)$ is injective. We know that the canonical homomorphism 
$\Pic(X)\rightarrow\ND(X)$ is injective by Lemma \ref{tor_free}, hence the 
homomorphism $\Pic(X)\rightarrow\Pic(D)$ is injective. 
\end{proof}

As a corollary, we get the following property which is important to classify 
higher dimensional log Fano manifolds with the log Fano pseudoindices $\geq 2$. 

\begin{corollary}\label{piccor}
\begin{enumerate}
\renewcommand{\theenumi}{\arabic{enumi}}
\renewcommand{\labelenumi}{(\theenumi)}
\item\label{piccor1}
Let $(X, D)$ be a log Fano manifold with $D\neq 0$ and 
the log Fano pseudoindex $\geq 2$. Then the restriction homomorphism 
$\Pic(X)\rightarrow\Pic(D_1)$ is injective 
for any irreducible component $D_1\subset D$.
\item\label{piccor2}
Let $\sX$ be an snc Fano variety with the snc Fano pseudoindex $\geq 2$. 
Then the restriction homomorphism 
$\Pic(\sX)\rightarrow\Pic(X_1)$ is injective for any 
irreducible component $X_1\subset\sX$.
\end{enumerate}
\end{corollary}

\begin{proof}
\eqref{piccor1}
We prove by induction on the dimension of $X$. 
If $\dim X\leq 2$, then the assertion is trivial by \cite[Proposition 4.1]{logfano}; 
we have $X\simeq\pr^2$ and $D$ is a hyperplane under the isomorphism. 

We can assume that the assertion holds for any log Fano manifold $(X', D')$ 
with $\dim X'=\dim X-1$. 
If $D$ is irreducible, then the assertion holds by 
Theorem \ref{mainthm} \eqref{mainthm1}. 
Let the irreducible decomposition $D=\sum_{i=1}^mD_i$ 
and let $D_{ij}:=D_i\cap D_j$ for any $i\neq j$; we can assume $m\geq 2$. 
We assume that an invertible sheaf $\sH$ on $X$ satisfies that 
$\sH|_{D_1}\simeq\sO_{D_1}$. It is enough to show that $\sH\simeq\sO_X$. 
We note that $(D_i, \sum_{j\neq i}D_{ij})$ is a log Fano manifold with the log Fano 
pseudoindex $\geq 2$ for any $1\leq i\leq m$. Hence the restriction homomorphism 
$\Pic(D_i)\rightarrow\Pic(D_{1i})$ is injective for any $2\leq i\leq m$ 
by the induction step. 
Thus $\sH|_{D_i}\simeq\sO_{D_i}$ for any $1\leq i\leq m$ since 
$(\sH|_{D_i})|_{D_{1i}}=(\sH|_{D_1})|_{D_{1i}}\simeq\sO_{D_{1i}}$ and the injectivity of the 
homomorphism $\Pic(D_i)\rightarrow\Pic(D_{1i})$ for $2\leq i\leq m$. 
Therefore $\sH|_D\simeq\sO_D$ by Proposition \ref{fanopicglue}; 
we remark that $D$ is an snc Fano variety. 
As a consequence, $\sH\simeq\sO_X$ holds by 
Theorem \ref{mainthm} \eqref{mainthm1}. 

\eqref{piccor2}
Let the irreducible decomposition $\sX=\bigcup_{i=1}^mX_i$ and let $X_{ij}:=X_i\cap X_j$ 
for any $i\neq j$; we can assume that $m\geq 2$. 
We assume that an invertible sheaf $\sL$ on $\sX$ satisfies that 
$\sL|_{X_1}\simeq\sO_{X_1}$. It is enough to show that $\sL\simeq\sO_{\sX}$. 
We note that $(X_i, \sum_{j\neq i}X_{ij})$ is a log Fano manifold with the log Fano 
pseudoindex $\geq 2$. Thus the restriction homomorphism 
$\Pic(X_i)\rightarrow\Pic(X_{1i})$ is injective for any $2\leq i\leq m$ by 
\eqref{piccor1}. We deduce that $\sL|_{X_i}\simeq\sO_{X_i}$
since $(\sL|_{X_i})|_{X_{1i}}=(\sL|_{X_1})|_{X_{1i}}\simeq\sO_{X_{1i}}$ and the injectivity 
of the homomorphism 
$\Pic(X_i)\rightarrow\Pic(X_{1i})$ for any $2\leq i\leq m$. Therefore we have 
$\sL\simeq\sO_{\sX}$ by Proposition \ref{fanopicglue}. 
\end{proof}

We can also show that the boundedness of Picard number for $n$-dimensional log Fano 
manifolds.

\begin{corollary}\label{bound_cor}
For any $n\in\Z_{>0}$, there exists $p(n)\in\Z_{>0}$ that satisfies the following 
conditions.
\begin{enumerate}
\renewcommand{\theenumi}{\arabic{enumi}}
\renewcommand{\labelenumi}{(\theenumi)}
\item\label{bound_cor1}
For any $n$-dimensional log Fano manifold $(X, D)$, the Picard number $\rho(X)$ of $X$ 
satisfies that $\rho(X)\leq p(n)$.
\item\label{bound_cor2}
For any $n$-dimensional snc Fano variety $\sX$, the Picard number $\rho(\sX)$ 
satisfies that $\rho(\sX)\leq p(n)$.
\end{enumerate}
\end{corollary}

\begin{proof}
For any snc Fano variety $\sX$, the rank of the Picard group $\rank(\Pic(\sX))$ is equal to the Picard number $\rho(\sX)$. It is easily shown 
since $H^1(\sX, \sO_\sX)=H^2(\sX, \sO_\sX)=0$ (see \cite[Corollary 2.26]{fujino}). 
We prove Corollary \ref{bound_cor} by induction on $n$. 

We know that if $(X, D)$ is a one-dimensional log Fano manifold, then $X\simeq\pr^1$. 
We also know that if $\sX$ is a one dimensional snc Fano variety, 
then $\sX$ is isomorphic to a reducible or smooth conic. These are proved 
in \cite[Example 2.10]{logfano}. Hence we can set $p(1):=2$. 

We assume that we can set $p(1),\dots,p(n-1)$. 
We know that there exists $q(n)\in\Z_{>0}$ 
such that $\rho(X)\leq q(n)$ for any $n$-dimensional Fano manifold $X$ 
by \cite[Theorem 0.2]{KMM}. We will show that we can set 
\[
p(n):=\max\{(n+1)(p(n-1)+1), q(n)\}.
\]
Let $(X, D)$ be an $n$-dimensional log Fano manifold. We can assume $D\neq 0$. 
We know that $\rho(X)\leq\rank(\Pic(D))+1$ by Theorem \ref{mainthm} and 
$D$ is an $(n-1)$-dimensional snc Fano variety by Theorem \ref{fujino_thm}. 
Therefore $\rho(X)\leq p(n-1)+1$. 

Let $\sX=\bigcup_{1\leq i\leq m}(X_i, D_i)$ be an $n$-dimensional snc Fano variety 
and the irreducible decomposition with the conductor. 
We know that $m\leq n+1$ by \cite[Corollary 2.21]{logfano} and each pair $(X_i, D_i)$
is an $n$-dimensional log Fano manifold by Remark \ref{irr_rmk}. Hence we have 
\[
\rho(\sX)\leq\sum_{i=1}^m\rho(X_i)\leq\max\{(n+1)(p(n-1)+1), q(n)\}
\]
by Proposition \ref{fanopicglue}. Therefore we have completed the proof of 
Corollary \ref{bound_cor}. 
\end{proof}

\section{Application to the Mukai conjecture}

In this section, we show that the Mukai conjecture 
(resp.\ the generalized Mukai conjecture) 
implies the log Mukai conjecture (resp.\ the log generalized Mukai conjecture), 
which stated in Section \ref{introsection}. 
First, we see the important example of $(r\rho-\rho+1)$-dimensional 
log Fano manifold with the log Fano inex $r$.

\begin{example}[Type $(\rho, r; m_1,\dots,m_{\rho-1})$]\label{lm}
Fix $r$, $\rho\geq 2$. Let $D\subset X$ be 
\begin{eqnarray*}
X & := & \pr_{(\pr^{r-1})^{\rho-1}}(\sO^{\oplus r}\oplus\sO(m_1,\dots,m_{\rho-1}))\\
D & := & \pr_{(\pr^{r-1})^{\rho-1}}(\sO^{\oplus r})
\end{eqnarray*}
with $m_1\dots,m_{\rho-1}\geq 0$, where the embedding $D\subset X$ is obtained by 
the canonical projection 
\[
\sO^{\oplus r}\oplus\sO(m_1,\dots,m_{\rho-1})
\rightarrow\sO^{\oplus r}.
\]
Then we have 
$\sO_X(-K_X)\simeq 
p^*\sO_{\pr_{(\pr^{r-1})^{\rho-1}}}(r-m_1,\dots,r-m_{\rho-1})\otimes\sO_\pr(r+1)$ 
and $\sO_X(D)\simeq 
p^*\sO_{\pr_{(\pr^{r-1})^{\rho-1}}}(-m_1,\dots,-m_{\rho-1})\otimes\sO_\pr(1)$, 
where $p\colon X\rightarrow\pr_{(\pr^{r-1})^{\rho-1}}$ is the projection 
and $\sO_\pr(1)$ is the tautological invertible sheaf on $X$ with respect to the 
projection $p$. 
It is easy to show that the invertible sheaf 
$p^*\sO_{\pr_{(\pr^{r-1})^{\rho-1}}}(1,\dots,1)\otimes\sO_\pr(1)$ is ample. 
Hence $(X, D)$ is an $(r\rho-\rho+1)$-dimensional log Fano manifold with the log Fano 
index $r$ and the log Fano pseudoindex $r$. 
\end{example}

We show from now on that the pair $(X, D)$ in Example \ref{lm} is the only example of 
$(r\rho-\rho+1)$-dimensional log Fano manifold with $D\neq 0$ and 
the log Fano index $r$ if we assume the low-dimensional Mukai conjecture 
and the low-dimensional log Mukai conjecture.

\begin{lemma}\label{technical}
Let $r$, $\rho\geq 2$. Consider a 
$\pr^r$-bundle $\pi\colon X\rightarrow(\pr^{r-1})^{\rho-1}$ 
and a divisor $D\subset X$ which satisfies that 
$D=(\pr^{r-1})^{\rho}$ and the restriction is the projection morphism 
$\pi|_D=p_{1,\dots,\rho-1}\colon D=(\pr^{r-1})^{\rho}\rightarrow(\pr^{r-1})^{\rho-1}$ 
and is a $\pr^{r-1}$-subbundle of $\pi$. 
If $(X, D)$ is a log Fano manifold of the log Fano pseudoindex $\iota\geq r$, 
then $(X, D)$ is isomorphic to the pair in 
Example \ref{lm} $($for some $m_1,\dots,m_{\rho-1}\in\Z_{\geq 0}$$)$. 
\end{lemma}

\begin{proof}
We can write the normal sheaf as 
$\sN_{D/X}=\sO_{(\pr^{r-1})^{\rho}}(-m_1\dots,-m_{\rho-1}, 1)$ 
such that $m_1,\dots,m_{\rho-1}\in\Z$. 

\begin{claim}\label{technical_claim}
We have $m_1,\dots,m_{\rho-1}\geq 0$. 
\end{claim}

\begin{proof}[Proof of Claim \ref{technical_claim}]
It is enough to show $m_1\geq 0$. 
Let $P=\pr^{r-1}$ be a general fiber of the projection 
$p_{2,\dots,\rho-1}\colon(\pr^{r-1})^{\rho-1}\rightarrow(\pr^{r-1})^{\rho-2}$ 
and let $X_P:=\pi^{-1}(P)$, $\pi_P:=\pi|_{X_P}\colon X_P\rightarrow P$ 
and $D_P:=X_P\cap D$. 
Then $(X_P, D_P)$ is also a log Fano manifold of the 
log Fano pseudoindex $\geq\iota\geq r$, 
the morphism $\pi_P$ is a $\pr^r$-bundle, $D_P=\pr^{r-1}\times\pr^{r-1}$, 
the restriction morphism 
$(\pi_P)|_{D_P}\colon\pr^{r-1}\times\pr^{r-1}\rightarrow\pr^{r-1}$ is 
the first projection and a $\pr^{r-1}$-subbundle of $\pi_P$. 
We also note that $\sN_{D_P/X_P}\simeq\sO_{\pr^{r-1}\times\pr^{r-1}}(-m_1, 1)$. 
Hence $X_P\simeq\pr_{\pr^{r-1}}(\sO_{\pr^{r-1}}^{\oplus r}\oplus\sO_{\pr^{r-1}}(m))$ 
with $m\geq 0$ and $D_P\simeq\pr_{\pr^{r-1}}(\sO_{\pr^{r-1}}^{\oplus r})$, where 
the embedding is obtained by the canonical projection 
\[
\sO_{\pr^{r-1}}^{\oplus r}\oplus\sO_{\pr^{r-1}}(m)\rightarrow
\sO_{\pr^{r-1}}^{\oplus r}
\]
under the isomorphism, by \cite[Theorem 4.3]{logfano}. 
Thus we can show that $\sN_{D_P/X_P}\simeq\sO_{\pr^{r-1}\times\pr^{r-1}}(-m, 1)$. 
Therefore we have $m_1=m\geq 0$. 
\end{proof}

The exact sequence 
\[
0\rightarrow\sO_{(\pr^{r-1})^{\rho-1}}\rightarrow
\pi_*\sO_X(D)\rightarrow(\pi|_D)_*\sN_{D/X}\rightarrow 0
\]
in \cite[Lemma 2.27 (i)]{logfano} splits since we know that 
\[
(\pi|_D)_*\sN_{D/X}\simeq\sO_{(\pr^{r-1})^{\rho-1}}(-m_1,\dots,-m_{\rho-1})^{\oplus r}
\]
by \cite[Lemma 2.28 (1)]{logfano} and by Claim \ref{technical_claim}. 
Therefore we have proved Lemma \ref{technical} by \cite[Lemma 2.27 (ii)]{logfano}. 
\end{proof}

\begin{thm}\label{mukai_log}
Conjectures $\M^n_\rho$ and $\LM^n_\rho$ 
$($resp.\ Conjectures $\GM^n_\rho$ 
and $\LGM^n_\rho$$)$ imply Conjecture $\LM^{n+1}_\rho$ 
$($resp.\ Conjecture $\LGM^{n+1}_\rho$$)$ for any $n$, $\rho\geq 2$. 
\end{thm}

\begin{proof}
We only prove that Conjectures $\GM^n_\rho$ and $\LGM^n_\rho$
imply Conjecture $\LGM^{n+1}_\rho$; 
the proof of the other assertion is essentially same. 

Let $(X, D)$ be an $(n+1)$-dimensional log Fano manifold with the log Fano 
pseusoindex $\iota$ and $D\neq 0$ which satisfies that 
$\rho(X)\geq\rho$ and $\iota\geq (n+\rho)/\rho$, 
where $n$, $\rho\geq 2$. 
Let $(D_1, E_1)\subset D$ be an arbitrary irreducible component of $D$ with 
the conductor divisor. Then $(D_1, E_1)$ is an $n$-dimensional log Fano manifold 
with the log Fano pseudoindex $\geq\iota$. 
We know by Corollary \ref{piccor} \eqref{piccor1} that $\rho(D_1)\geq\rho(X)\geq\rho$ 
since $\iota\geq 2$ holds. We note that $\iota>(n+\rho-1)/\rho$. 
Thus $E_1=0$ (hence $D=D_1$ is irreducible) holds by Conjecture $\LGM^n_\rho$. 
Hence we can apply Conjecture $\GM^n_\rho$ for $D$; we have 
$\rho(X)=\rho$, $\iota=(n+\rho)/\rho$ and $D\simeq(\pr^{\iota-1})^\rho$. 
We can assume $D=(\pr^{\iota-1})^\rho$. 

We run a $(-D)$-MMP which is also a $(K_X+D)$-MMP as in Notation \ref{logfano_mmp}. 
The restriction morphism $\pi_0|_D\colon D\rightarrow\pi(D)$ to its image is an 
algebraic space and is not a finite morphism by Propositions \ref{fund_prop} 
\eqref{fund_prop2} and \eqref{fund_prop3}. Thus $\dim\pi(D)<n$ since 
$D\simeq(\pr^{\iota-1})^\rho$. Hence $k=0$, that is, $\pi_0\colon X\rightarrow Y^0$ 
is of fiber type contraction morphism by Proposition \ref{fund_prop} \eqref{fund_prop1}. 
We can assume that $Y^0=(\pr^{\iota-1})^{\rho-1}$ and the restriction morphism 
$\pi_0|_{D}\colon D\rightarrow Y^0$ is equal to the projection morphism 
$p_{1,\dots,\rho-1}\colon (\pr^{\iota-1})^\rho
\rightarrow(\pr^{\iota-1})^{\rho-1}$ 
since $\rho(Y^0)=\rho-1$ and $\pi_0(D)=Y^0$ holds 
by Proposition \ref{fund_prop} \eqref{fund_prop1}. 
Let $[C]\in R^0$ be a minimal rational curve of $R^0$ on $X$. 
Then we have 
\begin{eqnarray*}
\iota-1 & = & \dim(\pi_0^{-1}(y)\cap D)\geq\dim\pi_0^{-1}(y)-1\\
 & \geq & (-K_X\cdot C)-2=(-(K_X+D)\cdot C)+(D\cdot C)-2\geq\iota-1
\end{eqnarray*}
for any closed point $y\in Y^0$ by Wi\'sniewski's inequality 
\cite{wisn} (see also \cite[Theorem 2.29]{logfano}). 
Thus we have $(-K_X\cdot C)=\iota+1$, $(D\cdot C)=1$ and 
$\dim\pi_0^{-1}(y)=\iota$ for any closed point $y\in Y^0$. 
Therefore the morphism $\pi_0\colon X\rightarrow Y^0$ is a $\pr^\iota$-bundle 
and the restriction morphism $\pi_0|_D\colon D\rightarrow Y^0$ is a 
$\pr^{\iota-1}$-subbundle of $\pi_0$. Hence Conjecture $\LGM^{n+1}_\rho$ holds 
by Lemma \ref{technical}. 
\end{proof}

\begin{corollary}\label{lgm_cor}
Conjecture $\LGM^n_\rho$ is true if 
\begin{enumerate}
\renewcommand{\theenumi}{\arabic{enumi}}
\renewcommand{\labelenumi}{(\theenumi)}
\item\label{lgm_cor1}
$\rho\leq 3$, or
\item\label{lgm_cor2}
$n\leq 6$. 
\end{enumerate}
\end{corollary}

\begin{proof}
It is an immediate corollary of Theorem \ref{mukai_log} and Remark \ref{intrormk}. 
\end{proof}

\medskip

\noindent K.\ Fujita

Research Institute for Mathematical Sciences (RIMS),
Kyoto University, Oiwake-cho, 

Kitashirakawa, Sakyo-ku, Kyoto 606-8502, Japan 

fujita@kurims.kyoto-u.ac.jp

\end{document}